\documentclass[12pt,oneside]{amsart} 
\usepackage{amssymb,amscd}
\usepackage{euscript}
 
      \theoremstyle{plain}
      \newtheorem{theorem}{Theorem}[section]
      \newtheorem{lemma}[theorem]{Lemma}
      \newtheorem{corollary}[theorem]{Corollary}
      \newtheorem{proposition}[theorem]{Proposition}
      \newtheorem{remark}[theorem]{Remark}
      
      \newtheorem{definition}[theorem]{Definition}        
          
\numberwithin{equation}{section}

      \makeatletter
      \def\@setcopyright{}
      \def\serieslogo@{}
      \makeatother

\def\A{\EuScript{A}} 
 
\def\E{\mathcal{V}}
\def\V{\mathcal{V}}
\def\k{M}
\def\M{{X}}

\def\R{\mathbb R}
\def\Z{\mathbb Z}
\def\N{\mathbb N}

\def\dist{\text{dist}}

\def\Id{\text{Id}}
\def\e{\epsilon}
\def\a{\alpha}
\def\ta{\tilde a}
\def\tb{\tilde b}
\def\b{\beta}
\def\la{\lambda}
\def\La{\Lambda}
\def\g{\gamma}

\def\r{\mathcal R}

\def\QED{\hfill\hfill{\square}}

\begin{document}

\date{July 23, 2016}
\author{Boris Kalinin and Victoria Sadovskaya$^{\ast}$}

\address{Department of Mathematics, The Pennsylvania State University, University Park, PA 16802, USA.}
\email{kalinin@psu.edu, sadovskaya@psu.edu}

\title[Periodic approximation of Lyapunov exponents]
{Periodic approximation of Lyapunov exponents for Banach cocycles} 

\thanks{$^{\ast}$ Supported in part by NSF grant DMS-1301693}


\begin{abstract} 

We consider group-valued cocycles over  dynamical systems.
The base system is a homeomorphism $f$ of a metric space satisfying a 
closing property, for example a hyperbolic dynamical system or a 
subshift of finite type. The cocycle $\A$ takes values in the group of invertible 
bounded linear operators on a Banach space and is H\"older continuous. 
We prove that upper and lower Lyapunov exponents of $\A$ with respect to 
an ergodic invariant measure $\mu$ can be approximated in terms of the 
norms of the values of $\A$  on periodic orbits of $f$.
We also show that these exponents
 cannot always be approximated by the exponents of $\A$ with respect
 to measures on periodic orbits. 
Our arguments include a result of independent interest on  construction
and properties of  a  Lyapunov norm for infinite dimensional setting.
As a corollary, we obtain estimates of the growth of the norm and of the 
quasiconformal distortion of the cocycle in terms of the growth at the
periodic points of $f$.

\end{abstract}

\maketitle 

\section{Introduction and statements of the results}

Cocycles with values in a group of linear operators on a vector space $V$
are the prime examples of non-commutative cocycles over dynamical systems.
For finite dimensional $V$ they  have been extensively studied, with examples including random matrices and derivative cocycles of  smooth dynamical
systems. 

The infinite dimensional case is more difficult and so far is much 
less developed. The simplest examples  are given by random and 
Markov sequences of operators. They correspond to locally constant cocycles 
over subshifts of finite type, which can be viewed as symbolic systems with hyperbolic behavior. Similarly to finite dimensional case, the derivative of a
smooth infinite dimensional system gives a natural example of an operator
valued cocycle. We refer to monograph \cite{LL} for an overview of results 
in this area and to \cite{GK,LY,M} for some of the recent developments.

In this paper we consider a  general setting of cocycles
of bounded operators on a Banach space and focus on H\"older continuous 
cocycles over dynamical systems with hyperbolic behavior. Some of our results
 hold for measurable cocycles over measure preserving systems, which are defined similarly.
  \newpage
  
\begin{definition} Let $f$ be a homeomorphism of a metric space $X$, 
let $G$ be a topological  group equipped with a complete metric $d$, and let 
$A$ be a function from $X$ to $G$  . 
The $G${\em -valued  cocycle over $f$ generated by }$A$ 
is the map $\A:\,\M \times \Z \,\to G$ defined  by 
 $$
\begin{aligned}
\A(x,0)=e_G,  \;\;\;\; &\A(x,n)=\A_x^n = 
A(f^{n-1} x)\circ \cdots \circ A(x), \\
&\A(x,-n)=\A_x^{-n}= (\A_{f^{-n} x}^n)^{-1} , \quad n\in \N.
\end{aligned}
$$
Clearly, $\A$ satisfies the {\em cocycle equation}\,
$\A^{n+k}_x= \A^n_{f^k x} \circ \A^k_x$.
\vskip.05cm
We say that the cocycle $\A$ is {\em bounded}\, if its generator $A $
 is a bounded function and that $\A$ is  {\em $\a$-H\"older} if $A$  is H\"older continuous with exponent $0<\a \le 1$ with respect to the metric $d$, i.e. there exists $M>0$ such that
 \begin{equation} \label{holder}
 d (A(x), A(y)) \le  \, M \dist (x,y)^\alpha \quad\text{for all sufficiently close }x,y \in X. 
\end{equation}
\end{definition}

In this paper we consider cocycles with values in the group of invertible 
operators on a Banach space $V$.  
The space
$L(V)$  of bounded linear operators on $V$  is a Banach space equipped 
with the operator norm $\|A\|=\sup \,\{ \|Av\| : \,v\in V , \;\|v\| \le 1\}.$ 
The open set $GL(V)$  of invertible elements in  $L(V)$ is a topological group  
and a complete metric space with respect  to the metric
$$
d (A, B) = \| A  - B \|  + \| A^{-1}  - B^{-1} \|.
$$
We call a $GL(V)$-valued cocycle $\A$ a {\em Banach cocycle}. 

\begin{definition}
Let $\mu$ be an ergodic $f$-invariant Borel probability measure on $\M$. The {\em upper and lower Lyapunov exponents} of $\A$ with respect to $\mu$ are
\begin{equation} \label{exponents}
\begin{aligned}
&\lambda(\A,\mu)= \lim_{n \to \infty} \frac 1n \log \| \A_x ^n \| 
\quad \text{for } \mu \text { almost every} \; x\in \M  , \\
&\chi (\A,\mu) = \lim_{n \to \infty} \frac 1n \log \| (\A_x ^n)^{-1} \|^{-1} 
\quad \text{for } \mu \text { almost every} \; x\in \M.
\end{aligned}
\end{equation}
\end{definition}

\noindent Each of these limits exists and is the same $\mu$ 
almost everywhere by the Subadditive Ergodic Theorem, see Section 2.
If $V$ is finite dimensional, these are precisely the largest and smallest
of the Lyapunov exponents given by the Multiplicative Ergodic Theorem 
of Oseledets.  

Our goal is to obtain a periodic approximation for the exponents of a cocycle
over a system with hyperbolic behavior. To streamline the arguments we 
formulate explicitly the property that we will use.

\begin{definition} \label{closing}
We say that a homeomorphism $f$ of a metric space $X$ 
satisfies the {\em closing property} if there exist 
constants $D,\, \g,\, \delta_0 >0$ such that for any $x \in X$ and $k\in\N$ with 
$\dist (x, f^k x) < \delta_0$ there exists a point $p \in X$ with 
$f^k p =p$ such that the orbit segments $x, fx, ... , f^k x$ and 
$p, fp, ... , f^k p$ are exponentially close, more precisely, 
\begin{equation}  \label{d-close-traj}
\dist (f^i x, f^i p) \le D \,  \dist (x, f^k x) \, e^{ -\g\, \min\,\{\,i,\,k-i \,\} }
\quad\text{for every }i=0, ... , k.
\end{equation}
\end{definition}

Systems satisfying this closing property include symbolic dynamical 
systems such as subshifts of finite type as well as smooth hyperbolic systems
such as hyperbolic automorphisms of tori and nilmanifolds, Anosov 
diffeomorphisms, and hyperbolic maps of locally maximal sets. For the smooth systems the closing property follows from Anosov Closing Lemma, see e.g. \cite[6.4.15-17]{KH}. 

Now we state our main result. Note that we do not assume compactness of $X$.

\begin{theorem} \label{main} 
Let $X$ be  separable metric space, let $f$ be a homeomorphism  of $X$
satisfying the closing property, let $\mu$ be an ergodic $f$-invariant Borel probability measure on $X$, and let $\A$ be a bounded H\"older continuous Banach 
cocycle over $f$.
Then for each  $\e>0$ there exists a periodic point $p=f^kp$ in $X$
such that 
\begin{equation} \label{main eq}
\left| \,\lambda(\A,\mu)-  \frac 1k \log \| \A_p ^k \| \, \right|<\e \quad\text{and}\quad
\left| \, \chi (\A,\mu)- \frac 1k \log \| (\A_p ^k)^{-1} \| ^{-1} \, \right|<\e.
\end{equation}
Moreover, for any $N\in \N$ there exists such $p=f^kp$ with $k>N$.
\end{theorem}

A periodic approximation of Lyapunov exponents for finite-dimensional $V$ 
was established in \cite[Theorem 1.4]{K} under slightly stronger closing assumption. It was shown, in particular,  that for each $\e >0$ there exists a  point $p=f^k p$ such that 
for the uniform measure $\mu_p$ on its orbit 
\begin{equation}\label{periodic approx}
\left| \,\lambda(\A,\mu)-  \lambda(\A,\mu_p) \, \right|<\e \quad\text{and}\quad
\left| \, \chi (\A,\mu)- \chi (\A,\mu_p) \, \right|<\e,
\end{equation}
and it is clear from the argument that \eqref{main eq} also holds.
We note that 
 $$\lambda(\A,\mu_p) = (1/k) \log \,(\text{spectral radius of }\A_p^k)
 $$
and so in our setting for  $p=f^kp$ satisfying \eqref{main eq}  we have
\begin{equation}\label{one-sided}
\la(\A,\mu_p)< \lambda(\A,\mu)+\e \quad \text{and}\quad
\chi(\A,\mu_p)> \chi(\A,\mu)-\e.
\end{equation}
However, approximation \eqref{periodic approx} is not always possible
in infinite-dimensional setting. The following proposition is based on an example by L. Gurvits of a pair of operators whose joint spectral radius is greater than the generalized spectral radius \cite{Gu}.

\begin{proposition} \label{no periodic}
There exists a locally constant cocycle $\A$ over a full shift on two symbols
and an ergodic invariant measure $\mu$ 
such that $
\,\la(\A,\mu) > \sup_{\mu_p} \la(\A,\mu_p)$, where the supremum is taken 
over all uniform measures $\mu_p$ on periodic orbits.
\end{proposition}

The proof of the finite dimensional periodic approximation result in \cite{K}
relies on Multiplicative Ergodic Theorem, which yields that the cocycle has
finitely many Lyapunov exponents and, in particular, the largest one is isolated. 
As this may not be the case in infinite dimensional setting even for a single operator,
we have to use a different approach. First we construct a measurable
Lyapunov norm for infinite dimensional setting and establish its temperedness.
This is a result of independent interest. We use this norm to obtain estimates of
the growth of the cocycle for trajectories close to ones with regular behavior. 
We also use further developments of the subadditive ergodic theorem 
obtained in \cite{KM,GK} and extend various techniques from \cite{K} and \cite{G}. 
Our methods may also be useful in the study of cocycles with values in 
diffeomorphism groups as well as of infinite-dimensional nun-uniformly 
hyperbolic dynamical systems on Hilbert or Banach manifolds.

\vskip.1cm

As a corollary of our main result, we obtain  estimates of the growth of the norm 
and of the {\it quasiconformal distortion} 
 $Q(x,n):= \| \A_x^n\| \cdot \| (\A_x^n)^{-1}\|$ of a cocycle $\A$ in terms of the growth  at periodic points.

\begin{corollary} \label{norm}
Let $f$ be a homeomorphism of a compact metric space $X$
satisfying the closing property and let $\A$ be a  H\"older continuous Banach 
cocycle over $f$. Then
$$
\text{\bf{(i)}}\;\; 
\lim _{n\to \infty} \sup \left\{ \| \A^n_x\|^{1/n}:\, x\in X \right\} = \,
  \limsup _{k\to \infty}\,\, \sup \left\{ \| \A^k_p\|^{1/k}:\,  p=f^kp \in X \right\}. \hskip2.7cm 
$$
In particular, if for some numbers $C$ and $s$ we have
$\|\A^k_p\| \le Ce^{s k}$ whenever $p=f^kp$, then 
for each $\e>0$ there exists a number $C_\e$ such that
$$  \|\A^n_x\| \le C_\e \,e^{(s+\e) n} \quad\text{for all }x\in X \text{ and }n\in \N.
$$
$$\text{\bf{(ii)}}\; \;
\lim _{n\to \pm \infty} \sup \left\{ Q(x,n)^{1/|n|}:\, x\in X \right\} = \,
  \limsup _{k\to \infty}\, \sup\, \left\{ Q(p,k)^{1/k}:\, p=f^kp \in X \right\}. \hskip2cm
$$
In particular, if for some numbers $C$ and $s$ we have
$Q(p,k) \le Ce^{s k}$ whenever $p=f^kp,$
then for each $\e>0$ there exists a number $C'_\e$ such that
$$
Q(x,n)\le C'_\e \,e^{(s+\e)|n|} \quad\text{for all }x\in X \text{ and }n\in \Z.
$$
\end{corollary}
The number $\hat \la (\A)=\lim _{n\to \infty} \sup \left\{ \| \A^n_x\|^{1/n}:\, x\in X \right\} $ gives the maximal growth rate of the cocycle and it is known that 
$\hat \la (\A) = \sup_\mu \lambda(\A,\mu)$, where the supremum is taken 
over all ergodic $f$-invariant measures. Part (i) of the corollary can
also be deduced from a recent work of M. Guysinsky \cite{G}, where he
 showed that for each $\e>0$ there exists a periodic point  $p=f^kp$ such 
 that $ \frac 1k \log \| \A_p ^k \| > \hat \la (\A) -\e$. He used this result to obtain
 a generalization of Livsic periodic point theorem for Banach cocycles.

Part (ii) of the corollary has a broad scope of applications as the  quasiconformal distortion $Q(x,n)$ 
is used in so called {\em fiber bunching} condition for cocycles over hyperbolic systems. This condition means that  $Q(x,n)$ is dominated by expansion and contraction in the base system and it plays a crucial role in most of the results on non-commutative cocycles over hyperbolic  systems. The corollary allows to obtain fiber bunching from the periodic data.

\begin{remark}
More generally, Theorem \ref{main} and Corollary \ref{norm} hold if we 
replace $X\times V$ by a H\"older continuous vector bundle $\V$ over 
$X$ with fiber $V$ and the cocycle $\A$ by an automorphism $\mathcal F: \V \to \V$ 
covering $f$. 
This setting is described in detail in Section 2.2 of \,\cite{KS13} and the proofs 
work without any significant modifications.
\end{remark}

\vskip.1cm

We discuss preliminaries on subadditive cocycles in Section 2, 
 give construction and properties of the Lyapunov norm in Section 3,
 prove Theorem \ref{main} in Section 4, and 
prove Proposition \ref{no periodic} and Corollary \ref{norm} in Section 5.


\section{Subadditive cocycles and their exponents}

A {\em subadditive cocycle} over a dynamical system $(X,f)$
is a sequence of functions $a_n:X\to \R\,$ such that 
$$
a_{n+k}(x)\le a_k(x) + a_n(f^kx) \quad \text{for all }x\in X \text{ and }k,n\in \N.
$$
\vskip.1cm

We define $\V=X\times V$ and view $\A^n_x$ as a map from $\V_x$ to $\V_{f^nx}$.
While this is not necessary for the trivial bundle, it makes notations and 
arguments more intuitive and facilitates extention to the bundle setting.
We denote 
\begin{equation}\label{an}
a_n(x)=\log \|\A_x^n\|, \quad
 b_n(x)=a_n(f^{-n}x)=\log \|\A_{f^{-n}x}^n\|,
\end{equation}
\begin{equation}\label{tan}
\ta_n(x)=\log \|(\A_x^n)^{-1}\|, \quad
 \tb_n(x)=\ta_n(f^{-n}x)=\log \|(\A_{f^{-n}x}^n)^{-1}\|=\log \|\A_{x}^{-n}\|.
\end{equation}

It is easy to see that  $a_n(x)$ and $\ta_n(x)$ are subadditive cocycles over $f$
and that   $b_n(x)$ and $\tb_n(x)$ are  subadditive cocycles over $f^{-1}$.   

For any ergodic measure-preserving transformation $f$ of a probability space 
$(X,\mu)$ and any subadditive cocycle over $f$ with integrable $a_n$, 
 the Subadditive Ergodic Theorem yields that for $\mu$ almost all  $x$ 
$$
 \lim_{n \to \infty} \frac 1n a_n(x)= \lim_{n \to \infty} \frac 1n a_n(\mu)=
\inf_{n \in \N} \frac 1n a_n (\mu) =: \nu (a,\mu) ,\; 
\text{ where }\; a_n(\mu)= \int_\M a_n(x) d\mu. 
$$
The limit $\nu (a,\mu) \ge -\infty$ is called the {\em exponent}\, of the cocycle $a_n$
with respect to $\mu$.

With our choice of $a_n(x)$, this theorem gives the existence 
of the limit in the first equation of \eqref{exponents} and that $\nu(a,\mu)$ is  
the upper Lyapunov exponent $\la=\la(\A,\mu)$ of the cocycle $\A$, so for  $\mu$ almost all  $x$
\begin{equation}\label{la}
\lim_{n \to \infty} \frac 1n \log \|\A_x^n\| =\lim_{n \to \infty} \frac 1n a_n(x)= 
\la:= \la (\A,\mu). 
\end{equation}
\vskip.2cm

Since  $\,b_n(\mu)=\int_\M b_n(x) d\mu = \int_\M a_n(x) d\mu =a_n (\mu),\,$  
it follows that for  $\mu$ almost all  $x$
\begin{equation}\label{bn}
 \lim_{n \to \infty} \frac 1n \log \|\A_{f^{-n}x}^n\| = 
 \lim_{n \to \infty} \frac 1n b_n(x) = \la.
\end{equation}
Similarly, for $\mu$ almost all $x$
\begin{equation}\label{-chi}
 \lim_{n \to \infty} \frac 1n \log \| \A_x ^{-n} \| =  \lim_{n \to \infty} \frac 1n \tb_n(x)
= \lim_{n \to \infty} \frac 1n \ta_n(x)= -\chi ,  
\end{equation}
where  $\chi:=\chi(\A,\mu)$ is the lower Lyapunov exponent of the cocycle $\A$
defined in \eqref{exponents}, and it is easy to see that $\chi \le \la$
and  both are finite.
We denote 
\begin{equation} \label{Lambda def}
\La = \La_\mu=\{ x\in X : \text{ equations \eqref{la}, \eqref{bn}, 
and \eqref{-chi} hold}\,\}
\end{equation}
and  in particular both equalities in \eqref{exponents} hold. Clearly, $\mu(\La)=1$.
\vskip.2cm


\section{Lyapunov norm} \label{Lyapunov norm}

In this section we construct a certain version of Lyapunov or adapted norm for 
our setting, which allows us to control the norms of $\A_x$ and $(\A_x)^{-1}$.
It is cruder than the usual Lyapunov norm for matrix cocycles constructed using Oseledets splitting, which is unavailable our setting.
Our construction is closer to that of an adapted metric for an Anosov system.
Since the Lyapunov norm in general depends only measurably on $x$, 
it is important to provide a comparison with the standard norm. We do
this in the second part of Proposition \ref{Lnorm properties} below.
\vskip.1cm

For a fixed $\e >0$ and a point $x\in \La$, the {\it $\e$-Lyapunov norm}  
$\|.\|_{x}=\|.\|_{x,\e}$  on $\V_x$ is defined as follows. For $u\in \V_x$, 
\begin{equation}  \label{Lprod}
\|u\|_{x} = \|u\|_{x,\e}  = \sum_{n=0}^\infty \, \|  \A_x^n (u)\| \,e^{-(\la +\e) n}
+\sum_{n=1}^\infty \, \|  \A_x^{-n} (u)\| \,e^{(\chi -\e) n}.
\end{equation}
By the definition \eqref{Lambda def} of $\La$, both series converge exponentially. 

\vskip.1cm
We denote the operator norm with respect to the Lyapunov norms  by 
$\| . \|_{y \leftarrow x}$.
 For any points $x,y\in \La$ and any 
linear map $A:\E_x \to \E_y$ it is defined by
$$
\| A \| _{y\leftarrow x}=\sup \, \{ \| A(u) \|_{y,\e}:  \; u\in \V_x, \;\, \| u\|_{x,\e}=1 \}.
$$

\begin{proposition} \label{Lnorm properties} 
Let $f$ be an ergodic invertible measure-preserving transformation 
of a probability space $(X,\mu)$ and let $\A$ be a bounded measurable 
Banach cocycle over $f$ with the upper Lyapunov exponent $\la$
and the lower Lyapunov exponent $\chi$. Then for each $\e>0$ the Lyapunov 
norm $\| . \|_{x,\e}$ given by \eqref{Lprod} satisfies the following.
\vskip.2cm

\noindent {\bf (i)\,} For each point $x$ in $\La$,
\begin{equation}  \label{estAnorm}
 \| \A_x \|_{f x \,\leftarrow x} \le e^ {\la+\e} \quad\text{and}\quad\,
 \| \A_x^{-1} \|_{f^{-1} x\, \leftarrow x} \le e^ {-\chi+\e}.
\end{equation}

\vskip.2cm

\noindent {\bf (ii)} There exists an $f$-invariant set $\r \subset \La$ with 
$\mu(\r)=1$ so that for each $\rho>0$ there exists a measurable 
function $K_{\rho}(x)$ such that  for  all $x\in \r$
\begin{equation}  \label{estLnorm}
\| u \| \le  \| u \|_{x,\e} \le K_{\rho}(x) \|u\| \quad \text{for all }
 u \in \E_x, \quad \text{and} 
\end{equation}
\begin{equation}  \label{estK}
 K_{\rho}(x) e^{-\rho |n|}  \le K_{\rho}(f^n x) \le  K_{\rho}(x) e^{\rho |n|}  \quad 
\text{for all } n \in \Z.
\end{equation}

\end{proposition}

\begin{proof} {\bf (i)} Let $u \in \V_x$. Using definition \eqref{Lprod}
and the fact that $\chi \le \la$  we obtain
$$
\begin{aligned}
& \|\A_x (u)\|_{fx} \, =\,  \sum_{n=0}^\infty \, \|  \A_{fx}^n (\A_x(u))\| \,e^{-(\la +\e) n}
\,+\, \sum_{n=1}^\infty \, \|  \A_{fx}^{-n} (\A_x(u))\| \,e^{(\chi -\e) n} \hskip2.5cm\\
&\, = \, \sum_{n=0}^\infty \, \|  \A_{x}^{n+1} (u)\| \,e^{-(\la +\e) n}
\,+\ \sum_{n=1}^\infty \, \|  \A_{x}^{-n+1} (u))\| \,e^{(\chi -\e) n} 
\end{aligned}
$$
$$
\begin{aligned}
& \, = \,e^{(\la+\e)} \left( \sum_{k=1}^\infty \, \|  \A_{x}^{k} (u)\| \,e^{-(\la +\e) k} 
+ \|u\| e^{(\chi -\e)-(\la+\e)}  +  \sum_{n=2}^\infty \, \|  \A_{x}^{-n+1} (u))\| \,e^{(\chi -\e) n-(\la+\e)} \right) \\
& \, \le \,e^{(\la+\e)} \left( \sum_{k=1}^\infty \, \|  \A_{x}^{k} (u)\| \,e^{-(\la +\e) k} 
+ \|u\|  +  \sum_{n=2}^\infty \, \|  \A_{x}^{-n+1} (u))\| \,e^{(\chi -\e) (n-1)}\right)\\
&\, = \,  e^{(\la+\e)} \left(  \sum_{k=0}^\infty \, \|  \A_{x}^{k} (u)\| \,e^{-(\la +\e) k}
 +    \sum_{k=1}^\infty \, \|  \A_{x}^{-k} (u))\| \,e^{(\chi -\e) k}\right)
 \,= \,  e^{(\la+\e)} \|u\|_x,
\end{aligned}
$$
and the first inequality follows. The second inequality is obtained similarly.

\vskip.3cm

{\bf (ii)} The uniform lower bound $\|u\|_{x,\e}\ge \|u\|$ follows immediately from the definition
of $\|u\|_{x,\e}$,
so it remains to establish the upper bound.
\vskip.1cm

By \eqref{Lambda def} for each $\e>0$ and $x\in \La$ there exists
$N_\e(x) \in \N$ such that 
\begin{equation} \label{A_n}
 \| \A_x ^n \| \le  e^{(\la +\e) n}  \quad\text{and}\quad 
 \| \A_x ^{-n} \| \le  e^{(-\chi +\e) n} \quad\text{for all }
n>N_\e(x). 
\end{equation}
Then for all $n\in \N$
\begin{equation} \label{A_n 2}
 \| \A_x ^n \| \le \k_\e(x) e^{(\la +\e) n}  \quad\text{and}\quad 
 \| \A_x ^{-n} \| \le \k'_\e(x) e^{(-\chi +\e) n} ,\quad\text{where}
\end{equation}
\begin{equation} \label{k_e}
\begin{aligned}
&\k_\e(x)=\max \,\{ \,\| \A_x ^n \| e^{-(\la +\e) n}:\; 0\le n\le N_\e(x) \,\}  \quad\text{and}\\
& \k_\e'(x)=\max \,\{ \,\| \A_x ^ {-n} \| e^{(\chi -\e) n}:\; 0\le n\le N_\e(x) \,\}.
\end{aligned}
\end{equation}
We note that $\, \k_\e(x)\ge  \| \A_x ^0 \| =1$ and,   since
$\| \A_x ^n \| e^{-(\la +\e) n} \le 1$ for all $n > N_\e$, 
$$
\k_\e(x)=\sup \,\{ \,\| \A_x ^n \| e^{-(\la +\e) n}:\; n\ge 0 \,\}.
$$
 It follows, in particular, that the function $\k_\e$ is measurable.
 Similarly, $\k'_\e$ is also measurable.
Using \eqref{A_n 2} with $\e/2$, we estimate $\|u\|_{x,\e}$ as follows.
$$
\begin{aligned}
\|u\|_{x,\e}&  =\,
   \sum_{n=0}^\infty \,  \|  \A_x^n (u)\| \,e^{-(\la +\e) n} 
   + \sum_{n=1}^\infty \, \|  \A_x^{-n} (u)\| \,e^{(\chi -\e) n} \,\le\, \\
& \le\, \sum_{n=0}^\infty \,\|u\|  \, \k_{\e/2}(x) e^{(\la +\e/2) n} \,e^{-(\la +\e) n}
+ \sum_{n=1}^\infty \,\|u\|  \, \k_{\e/2}'(x) e^{(-\chi +\e/2) n} \,e^{(\chi -\e) n} \\
& =  \,\|u\|  \cdot ( \k_{\e/2}(x) +\k_{\e/2}'(x))/ (1-e^{-\e/2}) = :\,\|u\|  \cdot \tilde \k_\e (x).
\end{aligned}
 $$
By  Lemma \ref{tempered} below, the function  $\tilde \k_\e (x)$ is tempered on a set $\r$ of full measure. 
Hence by \cite[Lemma 3.5.7]{BP} for each $ \rho>0$,
 $$
 K_{\rho} (x) =\sum_{n\in \Z} \tilde \k_\e (f^n x) e^{-\rho |n|}
 $$
is a measurable function defined on $\r$ satisfying \eqref{estK}. The inequality 
\eqref{estLnorm} follows since $\tilde \k_\e(x) \le K_{\rho} (x) $.
\end{proof}


\begin{lemma} \label{tempered} 
For each $\e>0$ the functions $\k_\e$ and $\k_\e'$ defined by \eqref{k_e} are 
forward and backward tempered on a set $\r \subset \La$ with $\mu(\r)=1$, that is 
for all $x\in \r$
$$
 \text{\bf (i)} \;\; \lim_{n \to \infty} \frac 1n \log \k_\e(f^nx) =0 \quad\text{ and }\quad
  \text{\bf (ii)}  \;\; \lim_{n \to \infty} \frac 1n \log \k_\e(f^{-n}x) =0
$$
and similarly for $\k_\e'$.
\end{lemma}

\begin{proof}
We fix $\e>0$ and give the proof for $\k_\e$. The result for $\k_\e'$ follows
by reversing the time. We note that $\k_\e \ge 1$ so that the sequences in 
(i) and (ii) are nonnegative.
\vskip.1cm 

{\bf (i)} First we show that $\k_\e$ is forward tempered. We consider $s>0$ and suppose that for some $x\in \La$
\begin{equation}\label{>s}
\limsup \, \frac 1n \log \k_\e(f^nx) >s>0.
\end{equation}
Then there exist infinitely many $k\in \N$ such that $\k_\e(f^k x) > e^{s k}$, 
and for any such $k$ there exists $n=n(k)$ such that 
\begin{equation} \label{sk}
  \| \A_{f^k x} ^n \| = \k_\e(f^k x)e^{(\la +\e) n} >  e^{s k} e^{(\la +\e) n}.
\end{equation}
Since the cocycle $\A$ is bounded,
there exist constants $\lambda'>\la$ and  $\chi' < \chi$ such that 
\begin{equation}\label{prime}
 \| \A_x \| \le e^{\lambda'} \; \text{ and}\quad  \| \A_x^{-1} \| \le e^{-\chi'}
 \;\text{ for all }x\in \M.
\end{equation}
First we show that if $k$ and $n$ satisfy inequalities \eqref{sk} and $k> N_{\e/2}(x)$,
where $N_{\e/2}(x)$ is defined as in \eqref{A_n}, then they must be comparable, more precisely,  
\begin{equation} \label{cC}
  ck < n < Ck ,\quad \text{where}\quad C=1+2(\la -\chi')/\e \quad 
  \text{and}\quad c=s/(\la' -\la).
\end{equation}
Indeed, for such $k$ inequality \eqref{A_n} implies  
$\| \A_x ^{n+k} \| \le  e^{(\la +\e/2) (n+k)}$ so we obtain 
$$
 \| \A_{f^k x} ^n \| = \| \A_x ^{n+k} \circ (\A_ x ^k)^{-1}\| \le
\| \A_x ^{n+k} \| \cdot \| (\A_ x ^k)^{-1}\| \le  e^{(\la +\e/2) (n+k)}\, e^{-\chi' k}.
$$
Together with \eqref{sk} this yields
$\,e^{s k} e^{(\la +\e) n}\le e^{(\la +\e/2) (n+k)}\, e^{-\chi' k}$ and so
$$
 n\le k\cdot 2(\la -\chi'-s+\e/2)/\e < Ck.
$$
On the other hand, \eqref{sk} and $\| \A_{f^k x} ^n \| \le e^{\lambda' n} $
imply $e^{s k} e^{(\la +\e) n} \le e^{\lambda' n} $ and so 
$$
 n\ge k\cdot s/(\la' -\la-\e)>ck.
$$

Now we show that there exists a set of full measure such that for each $x$ 
in this set \eqref{>s} leads to a contradiction.
We choose $\delta=\delta(s)$ so that  
\begin{equation}\label{delta}
0<\delta< c/(C+1),
\quad\text{where $c$ and $C$ are as in \eqref{cC}.}
\end{equation}
By Egorov's theorem, there exists a set $Y=Y(\delta)$ with $\mu(Y)> 1-\delta$ so that 
the convergence in \eqref{bn} is uniform on $Y$, and thus there exists 
an integer $N'_\e=N'_\e(Y)$ such that
\begin{equation}\label{N'_e}
b_n(y) < (\la +\e)n \quad\text{for all }y\in Y \text{ and }n >N'_\e.
\end{equation}
 Let $\La_Y$ be the set of full $\mu$ measure on which Birkhoff Ergodic 
 Theorem holds for the indicator function of $Y$, i.e. for all $x \in \La_Y$ 
$$
\lim_{N \to \infty} \frac 1N\,\# \{\ell:  1 \le \ell  \le N \text{ and } f^\ell x \in Y \} = \mu(Y)> 1-\delta.
$$
Now we consider $x \in \La \cap \La_Y$. Then for all sufficiently large  $k$ we have
$$
 \# \{\ell:  1 \le \ell  \le (k+n) \text{ and } f^\ell x \notin Y \} < \delta (k+n).
$$
Hence there exists $n'\le n$ with $n-n'< \delta (k+n)$ such that $y=f^{k+n'}x \in Y$.
Using \eqref{cC} we get 
\begin{equation}\label{n'}
n-n'< \delta (k+n) <\delta (k+Ck) =\delta (C+1) k .
\end{equation}
Using again \eqref{cC}  we obtain
$$n'= n - (n-n') > ck -  \delta (k+Ck)=(c-\delta (C+1))k,
$$  
where $c-\delta (C+1)>0$ by \eqref{delta}.
Therefore, since $k$ can be chosen arbitrarily large, we can assume that $n'>N'_\e$.
Then, as $y=f^{k+n'}x \in Y$, we can estimate
$$
\|\A^n_{f^kx}\| \le \|\A^{n'}_{f^kx}\| \cdot \|\A^{n-n'}_{f^{k+n'}x}\| 
\le e^{b_{n'}(f^{k+n'}x)} e^{\la'(n-n')} < e^{(\la+\e)n'}  e^{\la'(n-n')}.
$$
Combining this with \eqref{sk}, we obtain
$$
   e^{s k} e^{(\la +\e) n}<\| \A_{f^k x} ^n \|  < e^{(\la+\e)n'}  e^{\la'(n-n')}
$$
so that $sk<(\lambda'-\la-\e)(n-n')$. Then using \eqref{n'}, \eqref{delta}, and
\eqref{cC}  we obtain a contradiction
$$
sk < (\lambda'-\la-\e)\delta k(C+1) < (\lambda'-\la) ck =sk.
$$
Thus for each $s>0$ there exists a full measure set $\La_Y=\La_Y(s)$ such that
$$
\limsup \frac1n \log \k_\e(f^n x) \le s\quad\text{for all }x\in \La \cap \La_Y,
$$
and the first part of the lemma follows by taking the intersection of $\La$ and $\La_Y (1/n)$, $n\in \N$.
\vskip.2cm

{\bf (ii)}  Now we show that $\k_\e$ is backward tempered using 
a similar argument. Suppose that for some $x\in \La$
$$
 \limsup \, \frac 1n \log \k_\e(f^{-n}x) >s>0.
$$
Then  $\k_\e(f^{-k} x) > e^{s k}$ for infinitely many $k\in \N$, 
and hence for some $n=n(k)$
\begin{equation} \label{sk 2}
  \| \A_{f^{-k} x} ^n \|  >  e^{s k} e^{(\la +\e) n}.
\end{equation}
By \eqref{bn} there exists $N'_\e(x)$ such that $b_k(x) < e^{(\la+\e)k}$ for  
$k>N'_\e(x)$. If $n>k>N'_\e(x)$ then we have 
$$
\begin{aligned}
   \| \A_{f^{-k} x} ^n \| & \le  \| \A_{f^{-k} x} ^k \| \cdot  \| \A_{ x} ^{n-k} \| \le
   e^{b_k(x)} \k_\e(x) e^{(\la+\e)(n-k)} \\
   & \le \k_\e(x) e^{(\la+\e)k} e^{(\la+\e)(n-k)}=
   \k_\e(x) e^{(\la+\e)n}, 
   \end{aligned}
$$
which is incompatible with \eqref{sk 2} for large $k$, and  hence $n\le k$.
Also, \eqref{sk 2} and $\| \A_{f^{-k} x} ^n \| \le e^{\lambda' n} $ imply 
$e^{s k} e^{(\la +\e) n} \le e^{\lambda' n} $ and so $n>ck$, where $c=s/(\la' -\la)$
as in \eqref{cC}. We conclude that for large $k$, any $n=n(k)$ in \eqref{sk 2} satisfies $ck < n \le k$.

\vskip.1cm 

We choose $\delta=\delta(s)$ so that $0<\delta< c$ and
take sets $Y=Y(\delta)$ with $\mu(Y)> 1-\delta$ and $\La_Y$ as in the first part. 
Then if $x \in \La \cap \La_Y$ and $k$ is sufficiently large there exists $n'\le n$ 
with $n-n'<\delta k$ such that $y'=f^{-k+n'} x\in Y$. Since $\delta<c$
and  $n'= n - (n-n') > ck -  \delta k= (c-\delta )k$ we conclude that $n'>N'_\e$
if $k$ is sufficiently large. Thus for such $k$ we can estimate as before
$$
\|\A^n_{f^{-k}x}\| \le \|\A^{n'}_{f^{-k}x}\| \cdot \|\A^{n-n'}_{f^{-k+n'}x}\| 
\le e^{b_{n'}(f^{-k+n'}x)} e^{\la'(n-n')} < e^{(\la+\e)n'}  e^{\la'(n-n')}.
$$
Combining this with \eqref{sk 2}, we obtain
$$
 e^{s k} e^{(\la +\e) n} <   \| \A_{f^{-k} x} ^n \|  <e^{(\la+\e)n'}  e^{\la'(n-n')}  ,
$$
which yields  a contradiction as $n-n'<\delta k$ and $\delta< c=s/(\la' -\la)$:
$$sk < (\lambda'-\la-\e)(n-n') < (\lambda'-\la) \delta k  =sk.
$$
Thus on the  full measure set $\La \cap \La_Y$ we have
$
\limsup \frac1n \log \k_\e(f^{-n} x) \le s,
$
and the second part of the lemma follows.
\end{proof}

\vskip.5cm


\section{Proof of Theorem \ref{main}}

\subsection{Preliminary results.}
We fix $\e>0$ and consider the corresponding Lyapunov norm 
$\| .\|_x=\|.\|_{x,\e}$ given by \eqref{Lprod}. 
We apply Proposition \ref{Lnorm properties} (ii)  with 
$\rho=\e$ and obtain a full measure set $\r\subset\La$  where the function
$K=K_\e$ satisfies \eqref{estLnorm} and \eqref{estK}.
For any $\ell >1$ we  define 
\begin{equation}  \label{Pset}
\r_\ell = \{x \in \r : \; \; K(x) \le \ell \},
\end{equation}
and note that $\mu (\r_\ell) \to 1$ as $\ell \to \infty$.

We recall that
$\| \A_x^n \|_{x_n \leftarrow x_0} \le  e^{n (\la + \e)}$ and
$\| (\A_x^n)^{-1} \|_{x_0 \leftarrow x_n} \le  e^{n (-\chi + \e)}$  for $x\in \La$.
In Lemma \eqref{mainest} we obtain similar estimates for {\em any} point $y\in X$ whose trajectory is close to that of a point $x\in \r _\ell$. Since the Lyapunov
norm may not exist at points $f^ny$ we will use the Lyapunov norms at the corresponding points  $f^nx$ for the estimates. Since the bundle $\V$ is trivial
this creates no problem. For a non-trivial bundle one would need to identify
spaces $\V_{f^n x}$ and $\V_{f^n y}$, which can be done for nearby points.

\begin{lemma} \label{mainest}
Let $f$ be an ergodic invertible measure-preserving transformation 
of a probability space $(X,\mu)$ and let $\A$ be a bounded  $\a$-H\"older 
Banach cocycle over $f$ with the upper Lyapunov exponent $\la$
and the lower Lyapunov exponent $\chi$. 

Then for any  $\g > \e /\a$ there exists a constant  $c=c(\A, \,\a \g - \e)$ such that 
for  any  point $x$  in $\r_\ell$ with $f^m x$ in $\r_\ell$
and any point $y \in X$ such that the orbit segments $x, fx, ... , f^m x$ and 
$\,y, fy, ... , f^m y$ satisfy with some $\delta>0$
\begin{equation}  \label{d-close}
\dist (f^i x, f^i y) \le \delta e^{ -\g\, \min\{i,\,m-i \} }
\quad\text{for every }i=0, ... , m
\end{equation}
we have for all $ \,0\le n\le m$
\begin{equation}  \label{d-close-coc}
  \| \A_y^n \| \le \ell\, \| \A_y^n \|_{x_n \leftarrow x_0} \le 
\ell \, e^{c \,  \ell \delta^\a} e^{n (\la + \e)} \quad\text{ and}
\end{equation}
\begin{equation}  \label{d-close-coc-1}
  \| (\A_y^n)^{-1} \| \le \ell e^{\e \min\{n,\,m-n \}} \| (\A_y^n)^{-1} \|_{x_0 \leftarrow x_n} \le 
\ell e^{\e \min\{n,\,m-n \}} e^{c \,  \ell \delta^\a} e^{n (-\chi + \e)}.
\end{equation}

\end{lemma}

\vskip.2cm

\begin{proof} 
First we prove \eqref{d-close-coc}. We denote 
$$
x_i=f^i x \quad\text{and}\quad y_i=f^i y,\quad  i=0, ... , m, 
$$
and
use \eqref{estAnorm} to estimate the Lyapunov norm for $0<n\le m$
\begin{equation} \label{mainest1}
\begin{aligned}
\| \A_y^n\|_{x_n \leftarrow x_0} \,& \le\,
\prod_{i=0}^{n-1} \| \A_{y_i} \|_{x_{i+1} \leftarrow x_i} \,\le\,
\prod_{i=0}^{n-1} \| \A_{x_i} \|_{x_{i+1} \leftarrow x_i} \cdot
\| (\A_{x_i})^{-1} \circ \A_{y_i} \|_{x_{i} \leftarrow x_i} \\
& \le \, e^{n (\la + \e)} \,
\prod_{i=0}^{n-1}  \| (\A_{x_i})^{-1} \circ \A_{y_i}  \|_{x_i \leftarrow x_i}. 
\end{aligned}
\end{equation}  
We consider $\Delta_i =  (\A_{x_i})^{-1} \circ \A_{y_i}  - \Id$.
Since $\A_x$ is $\a$-H\"older \eqref{holder} and $\| (\A_{x})^{-1} \|$ 
is uniformly bounded we obtain 
\begin{equation} \label{mainestDi}
\| \Delta_i \| \le \| (A_{x_i})^{-1} \| \cdot  \| \A_{y_i} - \A_{x_i} \| 
\le M' \dist (x_i, y_i) ^ \a \le M' \left( \delta e^{-\g \min \{i, m-i\} } \right)^ \a ,
\end{equation}
where the constant $M'$ depends only on the cocycle $\A$.

Since both $x$ 
and $f^m x$ are in $\r_\ell$ we have $K(x_i) \le \ell e^{\e \min\{i,m-i \}}$
by \eqref{estK} and \eqref{Pset}. 
Also, for any  points $x,y \in \r$ the   inequality \eqref{estLnorm}  yields 
\begin{equation}  \label{estMnorm}
 \| A \|_{y \leftarrow x} \le K (y) \| A \| \qquad \text{and} \qquad
\| A \| \le K(x)  \| A \|_{y \leftarrow x} .
\end{equation}
Using the first inequality we conclude that
$$
\| \Delta_i \|_{x_i \leftarrow x_i} \le  K(x_i) \|  \Delta_i \| \le \ell e^{\e \min\{i,m-i \}} \,   \| \Delta_i \|
\le \ell e^{\e \min\{i,m-i \}} \, M' \delta^ \a e^{-\g \a \min \{i, m-i\} } 
$$  
$$
\text{and} \quad \|(\A_{x_i})^{-1} \circ \A_{y_i} \|_{x_i \leftarrow x_i} \le 1 + \| \Delta_i \|_{x_i \leftarrow x_i} \le
1 +  M' \ell  \,  \delta^\a \, e^{(\e -\a \g) \, \min\{i,m-i \} } . 
$$   
Combining this with \eqref{mainest1}  we obtain
$$
\begin{aligned}
& \log (\| \A_y^n \|_{x_n \leftarrow x_0}) - n (\la + \e) \le\, 
\sum _{i=0}^{n-1} \log \, ( \| (\A_{x_i})^{-1} \circ \A_{y_i} \|_{x_i \leftarrow x_i}) \\
& \le M' \ell \delta^\a \sum _{i=0}^{n-1}  e^{(\e -\a \g) \, \min\{i,m-i \} }
\le M' \ell \delta^\a\cdot 2\sum _{i=0}^{\infty}  e^{(\e -\a \g)i}
= c \,  \ell \delta^\a
\end{aligned}
$$
since $\e < \a \g$. The constant $c$ depends only on the cocycle $\A$ 
and on $(\a \g - \e)$. We conclude using \eqref{estAnorm} that 
\begin{equation} \label{mainestL}
\| \A_y^n \|_{x_n \leftarrow x_0} \le e^{c \,  \ell \delta^\a} e^{n (\la + \e)}.
\end{equation} 
Since $K(x_0) \le \ell$ we can also estimate the standard norm 
using the second  inequality in \eqref{estMnorm} 
$$
 \| \A_y^n \| \le K(x_0) \| \A_y^n \|_{x_n \leftarrow x_0} \le 
\ell e^{c \,  \ell \delta^\a} e^{n (\la + \e)}. 
$$

The proof of \eqref{d-close-coc-1} is similar. The previous argument with 
$\A_x$ replaced by $(\A_x)^{-1}$ yields
$$
\| (\A_y^n)^{-1} \|_{x_0 \leftarrow x_n} \le  e^{c \,  \ell \delta^\a} e^{n (-\chi + \e)}.
$$
Then the standard norm can be estimated as 
$ \| (\A_y^n)^{-1} \| \le K(x_n) \| (\A_y^n)^{-1} \|_{x_0 \leftarrow x_n} $
and \eqref{d-close-coc-1} follows since $K(x_n) \le \ell e^{\e \min\{n,\,m-n \}} $.
\end{proof}

\vskip.1cm

In the proof of the theorem we will also use the following results by 
A. Karlsson and G. Margulis and by M. Guysinsky.

\begin{proposition}\cite[Proposition 4.2]{KM} \label{KM}
Let $a_n(x)$ be an integrable subadditive cocycle with exponent 
$\la >-\infty$ over an ergodic measure-preserving system $(X,f,\mu)$.
Then there exists a set  $E\subset X$ with $\mu(E)=1$ such that for each 
$x\in E$ and each $\e>0$ there exists an integer $L=L(x,\e)$ and infinitely 
many $n$  such that 
\begin{equation}\label{good n}
  a_n(x)-a_{n-i} (f^ix) \ge (\la-\e)i \quad\text{for all }\, i \text{ with } L\le i \le n.
\end{equation}
\end{proposition}
\vskip.2cm

\begin{lemma}\cite[Lemma 8]{G} \label{G}
Let $f:X\to X$ be a homeomorphism
preserving an ergodic Borel probability measure $\mu$. 
Then there exists a set P with $\mu(P)=1$ such that 
for each $x\in P$ and  $\e, \delta >0$ there exists an integer 
 $N=N(x,\e,\delta)$ such that 
 if $n > N$ then there is an integer $k$ with 
 $$
 n(1 + \e) < k < n(1 + 2\e) 
\quad\text{and}\quad  \dist(x, f^k x) < \delta.
$$
\end{lemma}
\vskip.2cm


\subsection{Finding $\,p\,$ such that 
$\,| \,\lambda(\A,\mu)- \frac1k \log \| \A_p^k \| \, |<\e$ }$\,$
\vskip.1cm

We fix $0<\e < \min\{1,\a\g/3\}$, the Lyapunov norm $\|.\|_{x,\e}$, the function
$K=K_\e$, and the sets $\r$ and $\r_\ell$ as before.
Without loss of generality we can assume that $\r \subset (E \cap P)$,
where the set $E$ is given by Proposition \ref{KM} with $a_n(x)=\log \|\A^n_x\|$
and $P$ is given by Lemma \ref{G}. We take $\e' = 3\e/(\a\g)$ and choose $\ell$ so that 
$\mu (\r_\ell)> 1- \e'/2$. 
 \vskip.1cm

First we describe the choice of the periodic point $p=f^kp$.
We fix a point $x \in \r_\ell$ for which the Birkhoff Ergodic Theorem holds for the
indicator function of $\r_\ell$:
\begin{equation}  \label{bet}
\lim_{n \to \infty} \frac 1n\,\# \{i:  1 \le i  \le n \text{ and } f^\ell x \in \r_\ell \} 
= \mu(\r_\ell)> 1- \e'/2.
\end{equation}
We fix $L=L(x,\e)$ given by Proposition \ref{KM}.
We then take $\delta>0$ sufficiently small so that \eqref{Cdelta} is satisfied.  
We fix $N=N(x,\e',\delta/D)$ given by Lemma \ref{G}, where $D$ is as in 
the closing property, Definition \ref{closing}.
By Proposition  \ref{KM} there are
arbitrarily large $n$ satisfying \eqref{good n}. We consider such an $n$ greater
than $N$ and $L$. Then by Lemma \ref{G} there exists $k$ such that 
 $n(1 + \e') < k < n(1 + 2\e')$ and $\dist(x, f^k x) < \delta/D$. Then by
 the closing property there exists a periodic point $p=f^kp$ such that
 \begin{equation}  \label{x-p}
\dist (f^i x, f^i p) \le \delta e^{ -\g\, \min\{i,\,k-i \} }
\quad\text{for every }i=0, ... , k.
\end{equation}
By \eqref{bet}, if $n$ is sufficiently large then there exists $m$
 such that $f^m x \in \r_\ell$ and $n \le m \le n(1 + \e') <  k$.
We summarize our choices:

$$
 \begin{aligned}
& \max\{L,N\} < n \le m \le n(1 + \e') <  k < n(1 + 2 \e'), \\ 
& \, n \text{ satisfies \eqref{good n}}, \quad x,f^m x \in \r_\ell, \quad p=f^k p \text{ satisfies \eqref{x-p}}.
\end{aligned}
 $$
 
First we obtain an upper estimate for $\| \A^k_p\|$. 
Since \eqref{x-p} also holds with $m$ in place of $k$, we can apply
 Lemma \ref{mainest} with  $y=p$ to get
  $$
 \| \A_p^m \| \le \ell e^{c \,  \ell \delta^\a} e^{m (\la + \e)}. 
$$
As $\,k-m <2\e'n < 2\e' k,\,$ we have
 $$
 \begin{aligned}
 \| \A_p^k \| &\le  \| \A_p^m \|  \cdot \| \A_{f^m p}^{k-m} \|  \le 
 \ell e^{c \,  \ell \delta^\a} e^{m (\la + \e)} \cdot e^{\la'(k-m)} 
 =\ell e^{c \,  \ell \delta^\a} e^{k (\la + \e)+(\la'-\la-\e) (k-m) } \\
& \le \ell e^{c \,  \ell \delta^\a} e^{k (\la + \e)+(\la'-\la) 2\e' k }.
\end{aligned}
$$
Taking logarithm we obtain
$$
\frac1k \log \| \A_p^k \| \le 
\la + \e+ (\la'-\la)2\e' +\frac1k(\log \ell + c \,  \ell \delta^\a).
$$
Taking $n$, and hence $k$, sufficiently large we conclude that 
\begin{equation}\label{upper}
\frac1k \log \| \A_p^k \| \le \la + \e+ (\la'-\la)2\e' +\e = \la+ 2\e+ 6\e(\la'-\la)/(\a\g).
\end{equation}

\vskip.2cm 

 Now we obtain the lower estimate for $\| \A^k_p\|$. First we bound $\|\A^n_x -\A_p^n \|$.
 $$
\begin{aligned}
\A^n_x -\A_p^n \,&=\, 
\,\A^{n-1}_{x_1}\circ (\A_x -  \A_p) + (\A^{n-1}_{x_1}- \A^{n-1}_{p_1})\circ \A_p \\
& =\, \A^{n-1}_{x_1}\circ (\A_x -  \A_p) + 
\left( (\A^{n-2}_{x_2} \circ (\A_{x_1} - \A_{p_1}) + 
 (\A^{n-2}_{x_2} - \A^{n-2}_{p_2})\circ \A_{p_1}\right) \circ \A_p \\
& = \,  \A^{n-1}_{x_1}\circ (\A_x -  \A_p) + 
\A^{n-2}_{x_2} \circ (\A_{x_1} - \A_{p_1})  \circ \A_p + 
 (\A^{n-2}_{x_2} - \A^{n-2}_{y_2})\circ \A^2_p \\
 & = \dots = \, \sum_{i=0}^{n-1} \,\A^{n-(i+1)}_{x_{i+1}} \circ (\A_{x_i}-\A_{p_i}) \circ \A^i_p.
\end{aligned}
$$
Hence we can estimate the norm as follows
\begin{equation}\label{e1}
 \| \A^n_x -\A_p^n \| \le \sum_{i=0}^{n-1} 
 \,\|\A^{n-(i+1)}_{x_{i+1}}\| \cdot \|\A_{x_i}-\A_{p_i}\| \cdot \|\A^i_p\|.
\end{equation}
 Since $n$ satisfies \eqref{good n} of Proposition \ref{KM}
 with $a_n(x)=\log \|\A^n_x\|$, 
 $$
 a_{n-i}(x_i)\le a_n(x)- (\la-\e)i  \quad\text{for all }\, i \text{ with } L\le i \le n,
 $$
and thus for all such $i$
 $$ 
 \|\A^{n-(i+1)}_{x_{i+1}}\| \le \|\A^n_x\|\, e^{-(i+1)(\la -\e)}.
 $$
Using \eqref{x-p} and H\"older continuity of $\A$ we obtain   
$$
  \|\A_{x_i}-\A_{p_i}\| \le M\,\dist(x_i,p_i)^\a  
  \le M( \delta \, e^{ -\g\, \min\{i,k-i \}} ) ^\a  \le M \delta^\a \, e^{ -\a \g\, \min\{i,k-i \}}. 
 $$
To estimate the exponent we claim that 
$$
\a \g \min\{i,k-i \} \ge 3\e i \quad \text{for } i=1,...,n.
$$
 Indeed, this holds if $i=\min\{i,k-i \}$ as $\e<\a \g/3$. 
 If $k-i=\min\{i,k-i \}$ then 
$$
\a \g(k-i) \ge 3\e i \iff  (3\e + \a\g) i \le \a \g k \iff i \le \frac k{1+3\e/(\a\g)},
$$
which holds for $i \le n$ since $n<k/(1+\e')$ and $\e' = 3\e/(\a\g)$.
Thus we conclude that
\begin{equation}\label{decay}
  \|\A_{x_i}-\A_{p_i}\| \le M \delta^\a \, e^{ -3\e i}  \quad \text{for all } i=1, \dots, n.
\end{equation}
Applying Lemma \ref{mainest} as before, we get $\|\A^i_y\| \le \ell e^{c \,  \ell \delta^\a} e^{i (\la + \e)}$, for all  $i =1,..., m$.
Combining these estimates  we obtain that for $L \le i \le n$
$$
\|\A^{n-(i+1)}_{x_{i+1}}\| \cdot \|\A_{x_i}-\A_{p_i}\| \cdot \|\A^i_p\| \le
$$
$$ 
\le \|\A^n_x\|\, e^{-(i+1)(\la -\e)} \cdot M\delta^\a e^{-3\e i} \cdot
 \ell e^{c \,  \ell \delta^\a} e^{i(\la + \e)} =
 C_1(\delta) \|\A^n_x\|\, e^{-\e i}.
$$
where $C_1(\delta)=\delta^\a M \ell e^{c \,  \ell \delta^\a -\la+\e}$. 
We conclude that 
\begin{equation}\label{e2}
\begin{aligned}
& \sum_{i=L}^{n-1} 
\,\|\A^{n-(i+1)}_{x_{i+1}}\| \cdot \|\A_{x_i}-\A_{p_i}\| \cdot \|\A^i_p\|  \\
& \le \, C_1(\delta) \,\|\A^n_x\|\, \sum_{i=L}^{n-1}  e^{-\e i} 
\,\le\,  C_1(\delta) \,\|\A^n_x\|\, \frac{1}{1-e^{-\e} } \,=\, C_2(\delta) \|\A^n_x\|.
\end{aligned}
\end{equation}
where $ C_2(\delta)= C_1(\delta) (1-e^{-\e})^{-1}$.

For $i<L$ we estimate 
$\|\A^{n-i}_{x_{i}}\|  \le \|\A^n_x\| \cdot \| (\A_{x}^{i})^{-1}\| \le \|\A^n_x\|\, e^{-\chi' i}$,
where $\chi'<\chi$ is such that $\| (\A_x)^{-1}\| \le e^{-\chi'}$ for all $x\in X$.
Hence
\begin{equation}\label{e3}
\begin{aligned}
& \sum_{i=0}^{L-1}  
\|\A^{n-(i+1)}_{x_{i+1}}\| \cdot \|\A_{x_i}-\A_{p_i}\| \cdot \|\A^i_p\| \\
&\le \, \sum_{i=0}^{L-1} \|\A^n_x\|\, e^{-(i+1)\chi'} \cdot M\delta^\a e^{-3\e i} \cdot
 \ell e^{c \,  \ell \delta^\a} e^{i(\la + \e)}  \le C_3(\delta) \, \|\A^n_x\|,
\end{aligned}
\end{equation}
where $C_3(\delta)=Le^{-\chi'+(\la-\chi')L}\, \delta^\a M \ell e^{c \,  \ell \delta^\a }$, as $\la-\chi'>0$.

Combining  estimates \eqref{e1},  \eqref{e2} and  \eqref{e3} we obtain 
 $$
  \| \A^n_x -\A_p^n \| \,\le\, \|\A^n_x\|\, (C_2(\delta)   + C_3(\delta))
  \,\le \, \|\A^n_x\|/2
 $$
 since by the choice of $\delta>0$ we have
 \begin{equation}\label{Cdelta}
 C_2(\delta)   + C_3(\delta) =  
 \delta^\a M \ell e^{c \,  \ell \delta^\a} \left((1-e^{-\e})^{-1}e^{-\la +\e}+Le^{-\chi'+(\la-\chi')L}\right) < 1/2.
\end{equation}
Hence 
$$
 \|\A^n_p\| \,\ge\, \|\A^n_p\| - \| \A^n_x -\A_p^n \| \,\ge\, \|\A^n_x\|/2 
 > e^{(\la -\e)n}/2,
$$
provided that $n$ is sufficiently large for the limit in \eqref{la}. Since
$\A^n_p=(\A^{k-n}_{f^np})^{-1} \circ \A^k_p$,  
$$
\| \A^n_p\| \le \| (\A^{k-n}_{f^np})^{-1}\| \cdot \|\A^k_p\| \le \|\A^k_p\| \, e^{-\chi' (k-n)}.
$$
Hence 
$$
\|\A^k_p\|  \ge e^{\chi' (k-n)} \| \A^n_p\|   >  e^{(\la -\e)n+\chi' (k-n)}/2 >
e^{(\la -\e)k-(\la-\chi') (k-n)}/2,
$$
$$
\text{and so}\quad 
\frac 1k \log \|\A^k_p\| > \frac 1k [(\la -\e)k-(\la -\chi') (k-n) - \log 2].
$$
Since $k-n< 2\e' n <  2\e' k$ we obtain
$$
\frac 1k \log \|\A^k_p\| >  (\la -\e)-(\la -\chi') 2\e' - \log 2 /k> \la -2\e - 6\e(\la -\chi')/(\a\g)
$$
if $n$ and hence $k$ are sufficiently large.

We conclude that for each $\e>0$ there exists $p=f^kp$ satisfying 
this equation as well as \eqref{upper}, and the approximation of 
$\lambda(\A,\mu)$ by $\frac 1k \log \|\A^k_p\|$ follows.


\subsection{Finding $\,p\,$ approximating both $\lambda(\A,\mu)$ and
$\chi(\A,\mu)$} $\,$
\vskip.1cm

We describe how to modify the previous argument to obtain $p$ approximating
both the upper and the lower exponents. The construction of $p$  and the calculations are similar. The main difference is that for the point $x$ we need 
to find arbitrarily large $n$ satisfying both \eqref{good n}  and 
\begin{equation}\label{good n'}
  \ta_n(x)-\ta_{n-i} (f^ix) \ge (-\chi-\e)i \quad\text{for all }\, i \text{ with } L\le i \le n,
\end{equation} 
where $\ta_n(x)$ is from \eqref{tan}. For this we use an advanced version 
of Proposition \ref{KM} due to S. Gou\"ezel and A. Karlsson.

\begin{proposition}\cite[Theorem 1.1 and Remark 1.2]{GK} \label{GK} 
Let $a_n(x)$ be an integrable subadditive cocycle with exponent $\la>-\infty$ 
over an ergodic measure-preserving system $(X,f,\mu)$.
Then for each $\rho>0$ there exists a sequence $\e_i \to 0$ and 
a set  $E_\rho \subset X$
  with $\mu(E_\rho)>1-\rho$ such that for each $x\in E_\rho$ 
  the set $S$ of integers $n$ satisfying 
\begin{equation}\label{good n 2}
  a_n(x)-a_{n-i} (f^ix) \ge (\la-\e_i)i \quad\text{for all }\, i \text{ with } 1\le i \le n
\end{equation}
has asymptotic upper density is greater than $1-\rho$, that is
$$
 \overline{\text{Dens}}\,(S) \overset{\text{def}}{=}\, 
 \limsup \frac1N \left|\, S \cap [0, N-1] \,\right|  > 1-\rho.
$$ 
\end{proposition}
We take  $\rho<1/2$ and use Proposition \ref{GK} to obtain sets 
$E_\rho$ and $\tilde E_\rho$ of measure greater than $1-\rho$ for $a_n(x)$ 
and $\ta_n(x)$ respectively.  As in the previous argument, we take $x \in \r_\ell$ 
for which Birkhoff Ergodic Theorem holds for the indicator function of $\r_\ell$
and, in addition,  require that $x\in E_\rho \cap \tilde E_\rho$.
Then Proposition \ref{GK} ensures that there exist infinitely many $n$ 
satisfying \eqref{good n 2} simultaneously for $a_n(x)$ and for $\ta_n(x)$ with 
exponent $\chi$.  Moreover, there exists
$L=L(\e)$ so that $\tilde \e_i, \e_i < \e$ for all $i\ge L$ and thus we obtain that for such $n$ 
we have both \eqref{good n} and \eqref{good n'}.

All other choices remain the same except we take $\e' = 4\e/(\a\g)$. 
The argument for $\la$ is unchanged. The upper estimate for $-\chi$ is the same direct application of \eqref{d-close-coc-1} in Lemma \ref{mainest}. 
To obtain the lower estimate of $-\chi$ we use the equation
$$
 (\A^n_x)^{-1} -(\A_p^n)^{-1}  \,=\, 
  \sum_{i=0}^{n-1} \,(\A^i_p)^{-1}  \circ \left( (\A_{x_i})^{-1}-(\A_{p_i})^{-1} \right) 
  \circ (\A^{n-(i+1)}_{x_{i+1}})^{-1},
$$
which  yields an inequality for the norm  similar to \eqref{e1}. The first and third 
terms are then estimated using \eqref{d-close-coc-1} and \eqref{good n'}. 
We use the new choice of $\e'$ to get $e^{-4\e i}$ decay in \eqref{decay}
 to compensate for the extra term in \eqref{d-close-coc-1} compared to
  \eqref{d-close-coc}.
$\QED$

\section{Proofs of Proposition \ref{no periodic} and Corollary \ref{norm}}
Let $\mu_p$ be the uniform measure  on the orbit of a periodic point $p=f^kp$. 
Then 
$$
\begin{aligned}
&\la(\A,\mu_p)  = \lim_{N\to\infty} \frac{1}{N}\log \|\A_p^N\| \,=\,
\lim_{n\to\infty} \frac{1}{nk}\log \|(\A_p^k)^n\| \,= \, \\
& = \,\frac1k \,\log \left( \lim_{n\to\infty} \|(\A_p^k)^n\|^{1/n} \right)
\,= \, \frac1k \,\log \left( r(\A_p^k)\right)
\,\le\, \frac1k \log \| \A_p^k \|, 
\end{aligned}
$$
where $r(A)=\underset{n\to\infty}{\lim} \|A^n\|^{1/n} $ 
is the spectral radius of a linear operator $A$.\,
Similarly,
$$
-\chi(\A,\mu_p) = \frac1k \,\log \left( r((\A_p^k)^{-1}) \right), \;\text{ so }\;
\,\chi(\A,\mu_p) \ge \frac1k \log \| (\A_p^k)^{-1} \|^{-1} .
$$
It follows that for $p$ as in Theorem \ref{main} we have the 
one-sided estimates \eqref{one-sided}.


\subsection {Proof of Proposition \ref{no periodic}}$\,$
Suppose that $f$ is a homeomorphisms of a compact metric space $X$
and $a_n(x)$ is a continuous subadditive cocycle. Then it is easy to see
that $a_n=\sup _{x\in X} a_n(x)$ is a subadditive sequence and so there is limit
$$
\hat \nu (a):= \lim _{n \to \infty} a_n / n,
 \quad\text{where }\,a_n=\sup _{x\in X}\, a_n(x).
$$
By Theorem 1 in \cite{Sch}, $\,\hat \nu (a) = \sup _\mu \nu  (a, \mu)$ where supremum is taken over all $f$-invariant ergodic probability measures on $X$. 

Let $a_n(x)=\log\| \A^n_x \|$, where $\A$ is a H\"older continuous Banach 
cocycle over $f$, and let $f$ be  a homeomorphism of a compact metric space $X$ 
satisfying closing property. We denote $\hat \la (\A)= \sup_\mu \la (\A, \mu)$.
Then we have
$$
\hat \la (\A)  := \sup_\mu \la (\A, \mu) =\sup_\mu \nu (a, \mu) =  \hat \nu (a).
$$
We also consider the supremum  over all periodic measures $\mu_p$ and denote 
$$
\hat \la_p (\A)  := \sup_{\mu_p} \la (\A, \mu_p)=  \sup \left\{  \frac1k \log \left( r (\A_p^k)\right) : \;p= f^kp\right\}.
 $$
If $V$ is finite dimensional \cite[Theorem 1.4]{K} yields that $
\hat  \la (\A) =  \hat \la _p (\A).$ However,
for infinite dimensional space $\hat \la _p(\A)$ can be smaller 
than $\hat \la (\A)$. 

Indeed, let $f:X\to X$ be  the full shift  on two symbols, i.e.
$$
X=\left\{\,\bar x =(x_n)_{n\in \Z}: x_n \in \{0,1\} \right\} \quad \text{and}\quad
f(\bar x)=(x_{n+1})_{n\in \Z}
$$ 
and let $\A$ be a cocycle with values in bounded operators on Hilbert space $\ell _2$
given by
$$
 \A_{\bar x}=B_0 \,\text{ if }\, x_0=0 \quad\text{and}\quad 
  \A_{\bar x}=B_1 \,\text{ if }\, x_0=1.
$$
Since any $\A_x^n$ is a product of  $B_0$ and $B_1$ of length $n$, and any such product is $\A_x^n$ for some $x\in X$,  we have
$$
\begin{aligned}
 e^{\hat \la (\A)} & = \, \lim _{n \to \infty} \,\sup \,\left\{ e^{a_n(x)/n}:\; x\in X \right\}=
  \lim _{n \to \infty}\, \sup \left\{ \| \A^n _x \| ^{1/n} :\; x\in X \right\} \\
 & = \,\lim _{n \to \infty} \, \sup \left\{ \| A_n \cdots A_1 \| ^{1/n}:\;
  A_i \in\{ B_0, B_1 \}  \text{ for } i=1, ..., n \right\} 
 \,=:\,\hat \rho (B_0, B_1). \hskip1cm
 \end{aligned}
 $$ 
 Also, as any product of  $B_0$ and $B_1$ of length $n$
can be realized as  $\A_p^n$ for some $p=f^np$
 $$
 \begin{aligned}
& e^{\hat \la_p (\A)} =\, 
 \sup \left\{  \left( r (\A_p^k)\right)^{1/k} : \;p= f^kp\right\} = \\
 & = \,\sup \left\{ \left(r ( A_k \cdots A_1)\right) ^{1/k}: \;k\in \N , \;A_i \in\{ B_0, B_1 \}\right\} =\\
& \overset{(\ast)}{=}\,\limsup _{k \to \infty} \, \sup \left\{ \left(r ( A_k \cdots A_1)\right) ^{1/k}:\,
A_i \in\{ B_0, B_1 \}  \text{ for } i=1, ..., k\right\}  =: \bar \rho (B_0, B_1).
 \end{aligned}
 $$ 
Equality $(\ast)$ holds since $r(A) =r(A^m)^{1/m}$ and so taking 
$ A_k \cdots A_1$ with $r ( A_k \cdots A_1) ^{1/k}$ close to the supremum 
and repeating it we obtain an arbitrarily long product with the same value.
The number  $\hat \rho (B_0, B_1)$ is called the {\em joint spectral radius}\, of $B_0$ and $ B_1$, and
  $\bar \rho (B_0, B_1)$ is called the {\em generalized spectral radius.}
 \vskip.1cm
 
  By Theorem A.1 in \cite{Gu}, for any $0<\a <\b$ there exist two isomorphisms 
 $B_0$ and $B_1$ of $\ell _2$ such that $\bar \rho (B_0, B_1) =\a$ and $\hat \rho (B_0, B_1) =\b$. Hence for the corresponding cocycle $\A$ we have
$\hat \la _p (\A) < \hat\la (\A).$
It follows that  there is an ergodic measure 
 $\mu$ so that $\la (\A, \mu) >  \sup_{\mu_p} \la (\A, \mu_p)$.
 $\QED$

\subsection{Proof of Corollary \ref{norm}}
We prove the second part, and the first one is obtained similarly. We denote 
$$
\begin{aligned}
& \hat \sigma (\A)  \,= \,\lim _{n\to  \infty} \sup 
\left\{Q(x,n)^{1/n}:\; x\in X \right\} ,\\
& \hat \sigma_p (\A)=  \,\limsup _{k\to \infty}\, \sup 
\left\{ Q(p,k)^{1/k}:\; p=f^kp \right\},\\
& q_n(x)  \,=\, \log Q(x, n)=\log \|\A^n_x\|+\log \| (A^n_x)^{-1}\|.
\end{aligned}
$$
Since $a_n(x)=\log \| \A_x^n\|$ and 
$\ta_n(x)=\log \| (\A_x^n)^{-1}\|$ are subadditive cocycles over $f$,  so is 
$ q_n(x)=a_n(x)+\ta_n(x)$. For its exponent $\nu (q,\mu)$ we have
$$
  \nu(q, \mu) = \nu(a,\mu)+\nu(\ta, \mu) =\la(\A,\mu) - \chi(\A,\mu).
 $$
 As in the proof of Proposition \ref{no periodic} we consider
$$
\hat \nu (q)  =  \sup_\mu  \nu(q, \mu) = \lim _{n \to \infty}   q_n/n ,
 \quad\text{where }\,q_n=\sup _{x\in X} \,q_n(x)  =\log \,\sup _{x\in X} Q(x,n).
$$
It follows from Theorem \ref{main} that for any $N\in \N$
$$
\begin{aligned}
\hat \nu (q) & = \, \sup_\mu  \nu(q, \mu)=\sup_\mu(\la(\A,\mu) - \chi(\A,\mu)) \le \\
& \le \,\sup \left\{  \frac1k \log \| \A_p^k \|+\frac1k \log \| (\A_p^k)^{-1} \| : \;p= f^kp, \; k\ge N \right\}= \\
& = \,\sup \left\{  \frac1k \log Q(p,k) : \;p= f^kp, \; k\ge N  \right\}.
\end{aligned}
$$
Since $\log \hat \sigma (\A) = \hat \nu (q) $
 it follows that $\hat \sigma (\A) \le \hat \sigma_p (\A)$. 
 The opposite inequality  is clear, and so 
 $\hat \sigma (\A)= \hat \sigma_p (\A)$.  
 \vskip.1cm
 
Suppose that  $Q(p,k) \le Ce^{s k}$ whenever $p=f^kp.$ Then we have
$$
  s\ge \log \hat \sigma_p (\A) =\log \hat \sigma (\A) = \hat \nu(q)
  = \lim _{n \to \infty}   q_n/n.
$$
It follows that
 for each $\e>0$ there exists $N\in \N$ such that 
$q_n \le (s+\e)n$  for all $n>N$ and hence $Q(x,n) \le e^{(s+\e)n}$
for all  $x\in X$  and $n>N$.
Taking 
$$
C_\e = \max\, \{\, Q(x,n): \; x\in X \text{ and }1\le n \le N\, \},
$$ 
we obtain
$\,Q(x,n)\le C_\e e^{(s+\e) n}$ for all $x\in X$ and $n\in \N$.
\vskip.1cm

The statements for negative $n$ follow as $Q(x,-n)=Q(f^{-n}x,n).$
Hence
$$
\sup 
\left\{Q(x,-n)^{1/|n|}:\; x\in X \right\} =\sup 
\left\{Q(x,n)^{1/|n|}:\; x\in X \right\} 
$$
and so the limit as $n\to -\infty $ equals $\hat \sigma (\A).$
$\QED$


\end{document}